\documentclass{article}

\pdfoutput=1

\usepackage{amsmath, amssymb, amsthm, enumerate, fullpage}
\usepackage{verbatim, enumitem, bbm, etoolbox}
\usepackage{mathtools}
\usepackage{tensor}
\usepackage{leftidx}
\usepackage{cancel}
\usepackage{tikz}

\newcommand{\acl}{\textrm{acl}}

\newcommand{\icl}{\textrm{icl}}

\newcommand{\pcl}{\textrm{cl}^p}
\newcommand{\pdim}{\textrm{dim}^p}

\newcommand{\z}{\overline}
\newcommand{\f}{\mathfrak}

\theoremstyle{plain}
\newtheorem{thm}{Theorem}

\newtheorem{lemma}[thm]{Lemma}

\newtheorem{fact}[thm]{Fact}

\numberwithin{thm}{section}

\numberwithin{subcase}{case}

\theoremstyle{definition}
\newtheorem{defn}[thm]{Definition}
\newtheorem{remark}[thm]{Remark}

\newtheorem{notation}[thm]{Notation}
\newtheorem*{theorem*}{Theorem}

\def\dnf{\mathrel{\raise0.2ex\hbox{\ooalign{\hidewidth
$\vert$\hidewidth\cr\raise-0.9ex\hbox{$\smile$}}}}}

\def\forks{\mathrel{\raise0.2ex\hbox{\ooalign{\hidewidth
$\cancel\vert$\hidewidth\cr\raise-0.9ex\hbox{$\smile$}}}}}

\def\dindep{\mathrel{\raise0.2ex\hbox{\ooalign{\hidewidth
$\vert\rlap{${}^d$}$\hidewidth\cr\raise-0.9ex\hbox{$\smile$}}}}}

\def\ndindep{\mathrel{\raise0.2ex\hbox{\ooalign{\hidewidth
$\cancel\vert\rlap{${}^d$}$\hidewidth\cr\raise-0.9ex\hbox{$\smile$}}}}}

\newcommand{\north}{\not\perp}

\newcommand{\finsubset}{\subseteq_\textrm{Fin}}

\newcommand{\Kfin}{K_{\z{\alpha}}}

\newcommand{\Sfin}{S_{\z{\alpha}}}

\begin{document}
\title{Countable Models of the theories of Baldwin-Shi hypergraphs and their regular types}

\author{Danul K. Gunatilleka\thanks{Partially supported
		by NSF grant DMS-1308546.}}
\date{\vspace{-5ex}}

\newcommand{\Addresses}{{
		\bigskip
		\footnotesize
		
		\textsc{Department of Mathematics, University of Maryland at College Park}\par\nopagebreak
		\textit{E-mail address}: \texttt{danulg@math.umd.edu}
		
	}}
	
	\maketitle
	\maketitle
	
	\begin{abstract}
		We continue the study of the theories of Baldwin-Shi hypergraphs from \cite{DG1}. Restricting our attention to when the rank $\delta$ is rational valued, we show that each countable model of the theory of a given Baldwin-Shi hypergraph is isomorphic to a generic structure built from some suitable subclass of the original class of finite structures with the inherited notion of strong substructure. We introduce a notion of dimension for a model and show that there is a an elementary chain $\{\f{M}_{\beta}:\beta<\omega+1\}$ of countable models of the theory of a fixed Baldwin-Shi hypergraph with $\f{M}_{\beta}\preccurlyeq\f{M}_\gamma$ if and only if the dimension of $\f{M}_\beta$ is at most the dimension of $\f{M}_\gamma$ and that each countable model is isomorphic to some $\f{M}_\beta$. We also study the regular types that appear in these theories and show that the dimension of a model is determined by a particular regular type. Further, drawing on the work of Brody and Laskowski, we use these structures to give an example of a pseudofinite, $\omega$-stable theory with a non-locally modular regular type, answering a question of Pillay in \cite{Pil1}. 
	\end{abstract}

	\section{Introduction}\label{sec:intro}
	
	Fix a finite relational language $L$ where each relation symbol has arity at least 2 and let $K_L$ be the class of finite structures where each relation symbols is interpreted reflexively and symmetrically. Fix a function $\z\alpha:L\rightarrow (0,1)\cap\mathbb{Q}$. Define a \textit{rank function} $\delta:K_L\rightarrow \mathbb{Q}$ by $\delta(\f{A})=|A|-\sum_{E\in L}\z\alpha(E)|E^{\f{A}}|$ where $|E^{\f{A}}|$ is the number of subsets of $A$ on which $E$ holds. Let $\Kfin=\{\f{A}\in K_L:\delta(\f{A'})\geq 0 \text{ for all }\f{A'\subseteq A}\}$. Given $\f{A,B}\in\Kfin$, we say that $\f{A\leq B}$ if and only if $\f{A\subseteq B}$ and $\delta(\f{A})\leq \delta(\f{A'})$ for all $\f{A\subseteq A'\subseteq B}$. The class $(\Kfin,\leq)$ forms a Fra\"{i}ss\'{e} class, i.e. $\Kfin$ has amalgamation and joint embedding under $\leq$. In \cite{BdSh}, Baldwin and Shi initiated a systematic study of the \textit{generic structures} constructed from various sub-classes $K^*\subseteq \Kfin$ where $(K^*,\leq)$ forms a Fra\"{i}ss\'{e} class. In particular they obtained the stability of the theory of the $(\Kfin,\leq)$ generic $\f{M}_{\z\alpha}$. We call $\f{M}_{\z\alpha}$ the \textit{Baldwin-Shi hypergraph for $\z\alpha$} and denote its theory by $\Sfin$. We continue their study from \cite{DG1}
	
	We begin in Section \ref{sec:prelim} by introducing preliminary notions that we will be using throughout this paper that closely follows Section $2$ of \cite{DG1}. In Section \ref{sec:ExistThm} we collect some known results that will be using throughout the rest of the paper. We also use the notion of an \textit{essential minimal pair} and state Theorem \ref{lem:InfMinPa}, a key result from \cite{DG1}, which provides the existence of essential minimal pairs. This result will play a significant part in the results that follow.
	
	In Section \ref{sec:OmegaStable} we begin by studying the countable models of $\Sfin$. In Theorem \ref{thm:CountableModelsAsGenerics}, we prove that certain countable models of $\Sfin$ can be obtained as a generic structure by considering a particular subclass of the class of finite substructures used to construct the Baldwin-Shi hypergraph with the naturally inherited notion of strong substructure. We then use this result, along with a notion of dimension for models, to prove Theorem \ref{thm:CountableSpectrum} which establishes that the countable spectrum is $\aleph_0$. In Theorem \ref{thm:CountableModelChain} we sharpen this result and show that for countable $\f{M,N}\models\Sfin$, if the dimension of $\f{M}$ is at most the dimension of $\f{N}$, then $\f{M}$ embeds elementarily into $\f{N}$. Thus the countable models of $\Sfin$ form an elementary chain $\{\f{M}_{\beta}:\beta<\omega+1\}$ with $\f{M}_{\beta}\preccurlyeq\f{M}_\gamma$ for $\beta\leq \gamma$ with each model of $\Sfin$ isomorphic to some $\f{M}_\beta$.
	
	In Section \ref{sec:RegularTypes}  we study the regular types of $\Sfin$. A key result is Theorem \ref{thm:SomeNuggetsAreRegular} which identifies certain types as being regular. In Theorem \ref{thm:DomEquivalence} we establish that a certain class of types are non-orthogonal. We also show that there is a regular type whose  realizations determine the dimension of a model that was introduced in Section \ref{sec:OmegaStable}. We show in Theorem \ref{thm:NuggetsAreNonLocMod}, that these types are in fact not locally modular. We end the section with Theorem \ref{lem:2/cHasPrWrGEQ2}, which establishes that a large class of types are not regular.
	
	In Section \ref{sec:OtherProp} drawing on the work of Brody and Laskowski, we observe that certain of these generic structures have pseudofinite theories. Thus we obtain pseudofinite $\omega$-stable theories with non-locally regular modular types. This answers a question of Pillay in \cite{Pil1} on whether all regular types in a pseudofinite stable theory are locally modular.   
	
	The author wishes to thank Chris Laskowski for all his help and guidance in the preparation of this paper.
	
\section{Preliminaries}\label{sec:prelim}

We work throughout with a finite relational language $L$ where \textit{each relation symbol $E\in L$ is at least binary}. Let $ar:L\rightarrow \{n:n\in\omega \text{ and } n\geq 2\}$ be a function that takes each relation symbol to its arity.

\subsection{Some general notions}\label{subsec:GenNotions}

We begin with some notation. 

\begin{notation}
	We let $K_L$ denote the class of all finite $L$ structures $\f{A}$ (including the empty structure), where each $E\in{L}$ is interpreted symmetrically and irrelexively in $A$: i.e. $\f{A}\in K_L$ if and only if for every $E\in L$, if $\f{A}\models E(\z{a})$, then  $\z{a}$ has no repetitions and $\f{A}\models E(\pi(\z{a}))$ for every permutation $\pi$ of $\{ 0,\ldots,n-1\}$. By $\z{K_L}$ we denote the class of $L$-structures whose finite substructures all lie in $K_L$, i.e. $\z{K_L}=\{\f{M}:\f{M}\text{ an } L-\text{structure and if } \f{A\finsubset M}, \text{ then }\f{A}\in K_L \}$. We write $X\finsubset Z$, $\f{X\finsubset Z}$ to indicate that $|X|$ is finite and we let $\emptyset$ denote the unique $L$-structure with no elements in it.
\end{notation}

We now introduce the class $\Kfin$ as a subclass of $K_L$.

\begin{defn}\label{defn:ClassKfin}
	Fix a function ${\z\alpha}:L\rightarrow (0,1)\cap\mathbb{Q}$. Given $\f{A}\in K_L$, $N_{E}(\f{A})$ will denote the \textit{number of distinct subsets} of $A$ on which $E$ holds positively inside of $\f{A}$. Define a function $\delta$ on $K_L$ by $\delta(\f{A})= \vert A \vert-\sum_{E\in{L}}{{\z\alpha}(E)N_{E}(\f{A})}$ for each $\f{A}\in K_L$. Let $\Kfin=\{\f{A}  \vert \delta(\f{A'})\geq{0}\text{ for all }\f{A'}\subseteq{\f{A}}\}. $
\end{defn}

We adopt the convention $\emptyset\in K_L$ and hence $\emptyset\in \Kfin$ as $\delta(\emptyset)=0$. It is easily observed that $\Kfin$ is closed under substructure. Further the rank function $\delta$ allows us to view both $K_L$ and $\Kfin$ as collections of weighted hypergraphs. We proceed to use the rank function to define a notion of strong substructure $\leq$. Typically the notion of $\leq$ is usually defined on $\Kfin\times \Kfin$. However, we define the concept on the broader class $K_L\times K_L$. This will allow us to make the exposition significantly simpler via Remark \ref{rmk:ConstrVerifSimpl}. 

\begin{defn}\label{defn:defStrong}
	Given $\f{A, B}\in{K_L}$ with $\f{A\subseteq B}$, we say that $\f{A}$ \textit{is strong in} $\f{B}$, denoted by $\f{A}\leq{\f{B}}$, if and only if $\f{A\subseteq{B}}$ and $\f{\delta{(A)}\leq{\delta{(A')}}}$ for all $\f{A\subseteq{A'}\subseteq{B}}$. 
\end{defn}

\begin{remark}\label{rmk:ConstrVerifSimpl}
	Let $\f{A,B}\in K_L$. Further assume that $\f{A\subseteq B}$ with $\f{A}\in \Kfin$. If $\f{A\leq B}$, then $\f{B}\in\Kfin$. (use $(1)$ of Fact \ref{lem:MonoRelRankOvBase}). 
\end{remark}

\begin{defn}\label{defn:PostiveRankInfStructures}
	By $\z\Kfin$ we denote the class of all $L$-structures whose finite substructures are all in $\Kfin$, i.e. $\z{\Kfin}=\{\f{M}:\f{M}\text{ an } L-\text{structure and if } \f{A\finsubset M}, \text{ then }\f{A}\in \Kfin\}$. 
\end{defn}

\begin{remark}\label{fact:ExtSmthFraClass}
	The relation $\leq$ on $K_L \times K_L$ is reflexive, transitive and has the property that given $\f{A,B,C}\in K_L$, if $\f{A\leq{C}}$, $\f{B\subseteq{C}}$ then $\f{A\cap{B}\leq{B}}$ (use $(1)$ of Fact \ref{lem:MonoRelRankOvBase}). The same statement holds true if we replace $K_L$ by $\Kfin$ in the above. Further for any given $\f{A}\in\Kfin$, $\emptyset\leq{\f{A}}$.   
\end{remark}  

The following definition extends the notion of strong substructure to structures in $\z{K_L}$:

\begin{defn}
	Let $\f{X}\in\z{K_L}$. For $\f{A\finsubset{X}}$, $\f{A}$ \textit{is strong in} $\f{X}$, denoted by $\f{A\leq X}$, if $\f{A\leq B}$ for all $\f{A\subseteq B \finsubset Z}$. Given $\f{A}'\in K_L$ an embedding $f:\f{A'\rightarrow{\f{X}}}$ is called a \textit{strong embedding} if $f(\f{A'})$ is strong in $\f{X}$.
\end{defn}

\begin{defn}\label{defn:BasicDel2} 
	Let $n$ be a positive integer. A set $\{\f{B}_i:i<n\}$ of elements of $\Kfin$ is \textit{disjoint over} $\f{A}$ if $\f{A}\subseteq{\f{B}_i}$ for each $i<n$ and $B_i\cap{B_j}=A$ for $i<j<n$. If $\{\f{B}_i:i<n\}$ is disjoint over $\f{A}$, then $\f{D}$ is the \textit{free join} of $\{\f{B}_i:i<n\}$ if the universe $D=\bigcup\{B_i:i<n\}$, $\f{B}_i\subseteq{\f{D}}$ for all $i$ and, there are no additional relations, i.e. $E^\f{D}=\bigcup\{E^{\f{B}_i}:i<n\}$ for all $E\in{L}$. We denote a free join by $\oplus_{i<n}\f{B}_i$ . In the case $n=2$ we will use the notation $\f{B}_0\oplus_{\f{A}}\f{B}_1$ for $\oplus_{i<2}\f{B}_i$. 
\end{defn}

\begin{fact}\label{lem:FullAmalg1}
	If $\f{B, C}\in{\Kfin}$, $\f{A=B\cap{C}}$, and $\f{\f{A\leq{B}}}$, then $\f{B\oplus_{A}C}\in{\Kfin}$ and $\f{C\leq{B\oplus_{A}C}}$. 
\end{fact}We now turn our attention towards constructing the generic structure for $(\Kfin,\leq)$. We use the convention that $\emptyset\in\Kfin$. It is also immediate that $\delta(\emptyset)=0$ and that $\Kfin$ is closed under substructure.

\begin{defn}\label{defn:GenericStr}
	A countable structure $\f{M}\in\z{\Kfin}$ is said to be the generic for $(\Kfin,\leq)$ if \begin{enumerate}
		\item $\f{M}$ is the union of an $\omega$-chain $\f{A}_0\leq\f{A}_1\leq\ldots$ with each $\f{A}_i\in \Kfin$. 
		\item If $\f{A,B}\in\Kfin$ with $\f{A\leq B}$ and $\f{A\leq{M}}$, then there is $\f{{B'}\leq {M}}$ such that $\f{{B}\cong_{A}{B'}}$. 
	\end{enumerate}
\end{defn}

\begin{fact}\label{fact:ExistGenStr}
	$(\Kfin,\leq)$ is a Fra\"{i}ss\'{e} class (i.e. $(\Kfin,\leq)$ satisfies joint embedding and amalgamation with respect to $\leq$) and a generic structure for $(\Kfin,\leq)$ exists and is unique up to isomorphism.
\end{fact}

This justifies the following definition:

\begin{defn}
	For a fixed $\z\alpha$ we call the generic for $(\Kfin,\leq)$ the \textit{Baldwin-Shi hypergraph for $\z\alpha$}.
\end{defn}

\subsection{Closed sets}\label{subsec:ClosedSets}

In this section we generalize the notion of strong substructure to substructures of arbitrary size by introducing the notion of a closed set. This will provide us with a useful tool for analyzing the various theories of Baldwin-Shi hypergraphs. 

\begin{defn}\label{defn:MinPair}
	Let $\f{A,B}\in{K_L}$. Now $\f{(\f{A,B})}$ is a \textit{minimal pair} if and only if $\f{A\subseteq{B}}$, $\f{A\leq{C}}$ for all $\f{A\subseteq{C}\subset{B}}$ but $\f{A\nleq{B}}$.
\end{defn}

Note that  $\f{(\f{A,B})}$ is a minimal pair if and only if $\f{A\subseteq{B}}$, $\f{\delta{(A)}\leq{\delta{(C)}}}$ for all $\f{A\subseteq{C}\subset{B}}$ but $\f{\delta(\f{B})<\delta{(A)}}$. 

\begin{defn}\label{defn:ClosedSet}
	Let $\f{Z}\in{\z{K_L}}$ and ${X\subseteq{Z}}$. We say ${X}$ \textit{is closed in} $\f{Z}$ if and only if for all ${A\finsubset{X}}$, if ${(\f{A,B})}$ is a minimal pair with ${B\subseteq{Z}}$, then ${B\subseteq{X}}$.      
\end{defn}

\begin{remark}\label{rmk:Closed=StrongForFinite}
	As any $\f{A,B,C}\in K_L$ with $\f{A\leq{C}}$ and $\f{B\subseteq{C}}$ satisfies $\f{A\cap{B}\leq{B}}$ (see Remark \ref{fact:ExtSmthFraClass}) an easy argument yields that given $\f{Z}\in K_L$ and $\f{A\finsubset Z}$, $\f{A\leq Z}$ if and only if $\f{A}$ is closed in $\f{Z}$.  
\end{remark}

It is immediate from the above definition that any $\f{Z}\in{\z{K_L}}$, $Z$ is closed in $\f{Z}$ and that the intersection of a family of closed sets of $\f{Z}$ is again closed. These observations justify the following definition: 

\begin{defn}
	Let $\f{Z}\in\z{K_L}$ and ${X\subseteq{Z}}$. The \textit{intrinsic closure} of ${X}$ in ${Z}$, denoted by ${\icl_{\f{Z}}(X)}$ is the smallest set ${X'}$ such that ${X\subseteq{X'}\subseteq{Z}}$ and ${X'}$ is closed in ${Z}$.       
\end{defn}

\subsection{Some basic properties of the rank function}

We start exploring the rank function $\delta$ in more detail. The following facts are easily obtained by routine computations involving the $\delta$ function.

\begin{defn}\label{defn:relRank}
	Given $\f{Z}\in \z{K_L}$ and ${A, B}\finsubset{Z}$, let ${\delta({B/A})=\delta{(BA)}-\delta{(A)}}$.
\end{defn}

\begin{fact}\label{lem:MonoRelRankOvBase}\label{lem:BasicDel3}\label{fact:rankAddOverBases}\label{rmk:relRank}\label{fact:PositiveDimension}
	Let $\f{Z}\in\z{K_L}$ be non-empty and let ${A,B,B_i,C \finsubset Z}$.
	\begin{enumerate}
		\item Let $A'=A\cap{B}$. Now $\delta({B/A'})\geq \delta({B/A})=\delta(AB/A)$. Further if $B,C$ are disjoint and freely joined over $A$, then $\delta(B/AC)=\delta(B/A)$. 
		
		\item  If $\{{B}_i: i<n\}$ is disjoint over ${A}$ and $Z=\oplus_{i<n}B_i$ is their free join over $A$, then $\delta({Z/A})=\sum_{i<n}\delta({B}_i/A)$. In particular, if ${A}\leq{{B}_i}$ for each $i<n$, then ${A}\leq{\oplus_{i<n}{B}_i}$.
		
		\item $ \delta({B}_1{B}_2\ldots{B}_k/{A})=\delta({B}_1/{A})+\sum_{i=2}^k\delta({B}_i/{AB}_1\ldots{B}_{i-1}) $
		
	\end{enumerate}

\end{fact}

\section{A collection of known results}\label{sec:ExistThm}

In this section we provide some key definitions and results. We let $c:=lcm\{q_E:E\in L\}$ where $\z\alpha(E)=\frac{p_E}{q_E}$ is in reduced form. We begin with some definitions and some notation.  

\begin{defn}\label{defn:AEAxioms}
	The theory $\Sfin$ is the smallest set of sentences insuring that if $\f{M}\models{\Sfin}$, then
	\begin{enumerate}
		\item $\f{M}\in{\z{\Kfin}}$, i.e. every finite substructure of $\f{M}$ is in $\Kfin$ 
		\item For all $\f{A\leq{B}}$ from $\Kfin$, every (isomorphic) embedding $f:\f{A\rightarrow{M}}$ extends to an embedding $g:\f{B\rightarrow{M}}$
	\end{enumerate}
\end{defn}


\begin{notation} Given  $\f{A}\in{K_L}$ with a fixed enumeration $\z{a}$ of $A$, we write $\Delta_\f{A}(\z{x})$ for the atomic diagram of $\f{A}$. Also for $\f{A},\f{B}\in K_L$ with $\f{A}\subseteq{\f{B}}$ and fixed enumerations $\z{a}, \z{b}$ respectively with $\z{a}$ an initial segment of $\z{b}$;  we let $\Delta_{\f{A,B}}(\z{x}, \z{y})$ the atomic diagram of $\f{B}$ with the universe of $\f{A}$ enumerated first according to the enumeration $\z{a}$.
\end{notation}

\begin{defn}\label{defn:ExtensionFormula}
	Let $\f{A,B}\in{K}$ and assume $\f{A\subseteq{B}}$. Let $\Psi_{\f{A,B}}(\z{x})=\Delta_\f{A}(\z{x})\wedge\exists{\z{y}}\Delta_{\f{(A,B)}}{(\z{x},\z{y})}$. Such formulas are collectively called \textit{extension formulas} (over $\f{A}$). A \textit{chain minimal extension formula} is an extension formula $\Psi_{\f{A,B}}$ where $\f{B}$ us the union of a minimal chain over $\f{A}$. 
\end{defn}

The following appears in many places, including \cite{BdSh}, \cite{VY}, \cite{FW}.

\begin{defn}\label{defn:DimFunc}\label{defn:RelDimFunc}
	Let ${A,B\finsubset\mathbb{M}}$. Then $d({A})=\min \{\delta({A'})  \vert {{A\subseteq{A'}\subseteq{N}}, A' \text{ is finite.}}\}$. Further $d({B/A})=d({AB})-d({B})$. If ${X\subseteq\mathbb{M}}$ is infinite, then $d({A/X})=\min \{d({A/X}_0)  \vert X_{0}\finsubset{X}\}$.  
\end{defn}

\begin{remark}\label{rmk:propOfd}
	For finite $A,B,C$, it is easily observed that $d(A/C)\geq 0$, $d(AB/C)=d(A/BC)+d(B/C)$ and that $d(A)=\delta(\icl(A))$. In more general contexts, if ${X\subseteq\mathbb{M}}$ is infinite, then $d({A/X}):=\inf \{d({A/X}_0)  \vert X_{0}\finsubset{X}\}$. However it is easily observed that we can replace the infimum with minimum in the definition of $d$ as for any $\f{A}\in\Kfin$, $\delta(\f{A})\in\{k/c : k\in\omega \}$.     
\end{remark}

We collect some key results about $\Sfin$ from various sources in the following. 

\begin{thm}\label{thm:QuantElim}\label{Cor:SfinComplete}\label{lem:ClosEqui}\label{thm:StableFin1}\label{lem:CompleteTypesOverSets}\label{lem:WeakElimImg}\label{thm:ForkOvClsdSets1}\label{thm:ForkOvClsdSets}
	\begin{enumerate}
		\item Every $L$-formula is $\Sfin$-equivalent to a boolean combination of chain-minimal extension formulas (see \cite{DG1}).
		\item The theory $\Sfin$ is complete and is the theory of the generic for $(\Kfin,\leq)$. (see \cite{IkKiTs} or \cite{DG1}). 
		\item $\Sfin$ is $\omega$-stable. (see \cite{BdSh}).
		\item Given any $\f{M}\models\Sfin$ and $X\subseteq M$, $X$ is algebraically closed in $M$ if and only if $X$ is intrinsically closed in $M$. (see \cite{BdSh}, \cite{FW} or \cite{DG1}).
		\item The theory $\Sfin$ has weak elimination of imaginaries, i.e. every complete type over an algebraically closed set in the home sort is stationary. (see \cite{BdSh}, \cite{DG1} or \cite{VY})
		\item Let $\f{M}\models\Sfin$ be $\aleph_0$-saturated and let ${A}$ be a finite closed set of $\f{M}$. Suppose that $\pi$ is a consistent partial type over $A$. If any realization $\z{b}$ of $\pi$ in $\f{M}$ has the property that  $\z{b}{A}$ is closed in $\f{M}$, then $\pi$ has a unique completion to a complete type $p$ over $A$ (see \cite{DG1}).
		\item  $\Sfin$ has \textit{finite closures}, i.e. given any $\f{N}\models{\Sfin}$ for any finite $\f{A\subseteq{N}}$, there exists $\f{C}\in\Kfin$ s.t. $\f{A\subseteq{C}\leq{N}}$.
		\item Let $\mathbb{M}$ be a monster model of $\Sfin$. For algebraically closed $X,Y,Z$ with $Z=X\cap{Y}$, $X\dnf_{Z}Y$ if and only if $XY$ is algebraically closed and $X,Y$ are freely joined over $Z$ if and only if for any finite $X_0\subseteq X, Y_0\subseteq Y$, $d({X_0/Z})=d({X_0/ZY_0})$ and $\acl{({X_0Z})}\cap{\acl{({Y_0Z})}} = \acl{(Z)}$ (or equivalently $\icl{({X_0Z})}\cap{\icl{({Y_0Z})}} = \icl{(Z)}$: see \cite{BdSh} or \cite{VY}). 
		\item Let $\mathbb{M}$ be a monster model of $\Sfin$. Then $\mathbb{M}$ is \textit{flat}, i.e. given a finite collection of finite closed sets $\{E_i:i\in I\}$, $\sum_{s\subseteq I}(-1)^{|s|}d(E_s)\leq 0$ where $E_s=\bigcap_{i\in S} E_i$ and $E_{\emptyset}=\bigcup_{i\in I}E_i$ (see Section $7$ of \cite{FW}). 
	\end{enumerate}	
\end{thm}

\subsection{Essential Minimal Pairs}

We now define \textit{essential minimal pairs}. We use them here to study various properties of forking.

\begin{defn}\label{defn:essentialMinPa}
	Let $\f{B}\in\Kfin$ with $\delta(\f{B})>0$. We call $\f{D}\in \Kfin$ with $\f{B\subseteq D}$ an \textit{essential minimal pair} if $(\f{B,D})$ is a minimal pair and for any $\f{D'\subsetneq D}$, $\delta(\f{D'/D'\cap{B}})\geq 0$.  
\end{defn}

The following, in more general form, appears in \cite{DG1} as Theorem 3.32. It will form the backbone of many of the results to follow,

\begin{thm}\label{lem:InfMinPa}\label{thm:OmitPrelim2}
	Let $\f{A}\in\Kfin$ with $\delta(\f{A})=k/c>0$. There are $\f{D}\in\Kfin$ such that $(\f{A,D})$ is an essential minimal pair and satisfies $\delta(\f{D/A})=-1/c$. 
\end{thm}

We immediately obtain the following useful lemma.

\begin{lemma}\label{lem:FinStrsInRatCase}%
	Let $k\in\omega$. Given any $\f{B}\in\Kfin$, there is some $\f{D}\in\Kfin$ such that $\f{D}\supseteq \f{B}$, $\delta(\f{D})=k/c$ and for any $\f{A\leq B}$ with $\delta(\f{A})\leq k/c$, $\f{A\leq D}$.  \end{lemma}

\begin{proof}
	Given $\f{B}$ take $\f{D}_0$ to be the free join of $\f{B}$ with a structure with $k+1$ many points with no relations among them over $\emptyset$. Note that $\f{B\leq D}_0$. Let $l=c\delta(\f{D}_0)-k$. Consider a sequence $\f{D}_0\subseteq\ldots \subseteq\f{D}_l$ where each $(\f{D}_i,\f{D}_{i+1})$ is an essential minimal pair with $\delta(\f{D}_{i+1}/\f{D}_1)=-1/c$. We claim that $\f{D}=\f{D}_l$ is as required. Fix any $\f{A\leq B}$ with $\delta(\f{A})\leq k/c$. We show by induction on $i<l$ that if $\f{A\leq D}_{i}$, then $\f{A\leq D}_{i+1}$. Clearly $\f{A\leq D}_0$ as $\f{A\leq B\leq D}_0$. Fix $i<l$ and consider any $\f{F}$ such that $\f{A\subseteq F\subseteq D}_{i+1}$. If $\f{F=D}_{i+1}$ then $\delta(\f{F})\geq k/c\geq \delta(\f{A})$ and so $\delta(\f{F/A})\geq 0$. On the other hand, if $\f{F\neq D}_{i+1}$, then, $\delta(\f{F}/{\f{D}_{i+1}\cap\f{F}})$ since $(\f{D}_i,\f{D}_{i+1})$ is an essential minimal pair and $\delta(\f{D}_{i}\cap\f{F}/\f{A})\geq 0$ as $\f{A\leq D}_i$. Thus $\delta(\f{F/A})=\delta(\f{F}/\f{D}_i\cap\f{F})+\delta(\f{F}\cap\f{D}_i/\f{A})\geq 0$ as required. 
\end{proof}

\section{Countable models of $\Sfin$}\label{sec:OmegaStable}

Our goal in this section is to study the countable models of $\Sfin$. We begin by defining a notion of dimension for (countable) models. We then show that this notion of dimension is able to categorize countable models up to both isomorphism and elementary embeddability. Recall that $c$ is the least common multiple of the denominators of the $\z\alpha_E$ (in reduced form).

\begin{defn}\label{defn:DimofModel}
	Let $\f{M}\models \Sfin$. Let $\f{A\leq M}$. We let $\dim(\f{M/A})= \max\{\delta(\f{B/A}): \f{A\leq{B}\leq{M}}\}$. If there is no maximum, i.e. given any $z>0$, there will be some $\f{B\leq{M}}$ with $\delta(\f{B/A})>z$, we let $\dim(\f{M/A})=\infty$. We write $\dim(\f{M})$ for $\dim(\f{M}/\emptyset)$. 
\end{defn}

\begin{defn}\label{defn:FixedDimAmalgClass}
	Fix an integer $k\geq 0$ and let $K_{k/c}=\{\f{A}:\f{A}\in{\Kfin} \text{ and } \delta(\f{A})=k/c\}$. Let $(K_{k/c},\leq)$ be such that $\leq$ is \textit{inherited by} $\Kfin$ i.e. $\f{A\leq B}$ for $\f{A,B}\in K_{k/c}$ if and only if for all $\f{A\subseteq B'\subseteq B}$ with $\f{B'}\in\Kfin$, $\f{A\leq B'}$   
\end{defn}

We begin with the following technical lemma:

\begin{lemma}\label{lem:DimWitness}
	Let $\f{A,B,C,D}\in\Kfin$ with $\f{A\leq B,C}$; $\delta(\f{C/A})\geq\delta(\f{B/A})$ and $\f{D}=\f{B\oplus C}$ the free join of $\f{B,C}$ over $\f{A}$. We can construct $\f{H}\in\Kfin$ such that $\f{A,B,C}\leq \f{H}$, $\f{D\subseteq H}$ and $\delta(\f{H/C})=0$. Further if $\delta(\f{B/A})=\delta(\f{C/A})$, the $\f{H}$ that was constructed has the property $\delta(\f{H/B})=0$.   
\end{lemma}

\begin{proof}
	This follows from an easy application of Lemma \ref{lem:FinStrsInRatCase} on $\f{D}$. 
\end{proof}

We now work towards showing that certain countable models of $\Sfin$ can be built as  Fra\"{i}ss\'{e} limits $(K_{k/c},\leq)$. In Theorem \ref{thm:CountableSpectrum} we show that these are in fact, all of the countable models up to isomorphism.

\begin{lemma}\label{lem:CountableModelsAsGenerics}
	For any fixed integer $k\geq 0$, $(K_{k/c},\leq)$, where $\leq$ is inherited from $\Kfin$ is a Fra\"{i}ss\'{e} class. 
\end{lemma} 

\begin{proof}
	Fix an integer $k\geq 0$ and consider $K_{k/c}$. Let $\f{A,B,C}\in K_{k/c}$. Note that for the purposes of proving amalgamation, we may as well assume $\f{B,C}$ are freely joined over $\f{A}$ and that $\f{A\leq B,C}$. Note that  $\delta(\f{B/A})=\delta(\f{C/A})=0$. The required statement follows by a simple application of Lemma \ref{lem:DimWitness} on $\f{B\oplus_{A} C}$. For joint embedding consider $\emptyset\leq\f{B,C}$. Note that  $\delta(\f{B}/\emptyset)=\delta(\f{C/\emptyset})=k/c$. Apply Lemma \ref{lem:DimWitness} on $\f{B\oplus_{\emptyset} C}$, the free join of $\f{B,C}$ \textit{over} $\emptyset$.	
\end{proof}

We now prove the following theorem. To account for the fact that $K_{k/c}$ is not closed under substructure, we tweak the second condition in the definition of the generic structure (Definition \ref{defn:GenericStr}) to: for every finite $\f{A\subseteq M}$, there exists $\f{B}\in K_{k/c}$ with $\f{A\subseteq B}$ and $\f{B\leq M}$.

\begin{thm}\label{thm:CountableModelsAsGenerics}
	Let $k$ be a fixed integer with $k\geq 0$. Let $\f{M}_{k/c}$ be the generic for the  Fra\"{i}ss\'{e} class $(K_{k/c},\leq)$ where $\leq$ is inherited from $\Kfin$. Now $\f{M}_{k/c}\models\Sfin$ and $\dim(\f{M}_{k/c})=k/c$.
\end{thm}

\begin{proof}
	Fix an integer $k\geq 0$.  From Lemma \ref{lem:CountableModelsAsGenerics}, it follows that $(K_{k/c},\leq)$ where $\leq$ is inherited from $\Kfin$ is a Fra\"{i}ss\'{e} class. Let $\f{M}_{k/c}$ be the $(K_{k/c},\leq)$ generic. Note that given $\f{B}\in\Kfin$, there is some $\f{D}\in K_{k/c}$ such that $\f{D\supseteq B}$ by Lemma \ref{lem:FinStrsInRatCase}. Thus it suffices to show that $\f{M}_{k/c}$ satisfies the extension formulas in $\Sfin$.
	
	Let $\f{A,B}\in\Kfin$ with $\f{A\leq B}$ and assume that $\f{A\finsubset M}_{k/c}$. As $\f{M}_{k/c}$ is the $(K_{k/c},\leq)$ generic, there is some $\f{C\leq M}_{k/c}$ with $\f{A\subseteq C}$ and $\delta(\f{C})=k/c$. By Fact \ref{lem:FullAmalg1}, we have that $\f{D}=\f{B}\oplus\f{C}$, the free join of $\f{B,C}$ over $\f{A}$ is in $\Kfin$ and that $\f{C\leq D}$. Now using Lemma \ref{lem:FinStrsInRatCase}, we can find $\f{G}\in K_{k/c}$ such that $\f{D\subseteq G}$ and $\f{C\leq G}$. But as $\f{M}_{k/c}$ is the $(K_{k/c},\leq)$ generic we can find a strong embedding of $\f{G}$ into $\f{M}_{k/c}$ over $\f{C}$. Thus it follows that $\f{M}_{k/c}\models\forall\z{x}\exists{\z{y}}(\Delta_{A}(\z{x})\wedge\Delta_{A,B}(\z{x},\z{y}))$. Hence it follows that $\f{M}_{k/c}\models\Sfin$. Further as noted above, given any $\f{A\finsubset M}_{k/c}$, there is some $\f{C\leq M}_{k/c}$ with $\f{A\subseteq C}$ and $\delta(\f{C})=k/c$. Hence $\dim(\f{M}_{k/c})=k/c$.
\end{proof}

We now work towards classifying the countable models of $\Sfin$ up to isomorphism using our notion of dimension. 

\begin{lemma}\label{lem:DimAndStrngEm}
	Let $\f{M}\models\Sfin$ and $\f{A\leq M}$ be finite. Let $\f{D}\in{\Kfin}$ be such that $\f{A\leq{D}}$. Then $\dim(\f{M/A})\geq{\delta(\f{D/A})}$ if and only if there is some $g$ such that $g$ strongly embeds $\f{D}$ into $\f{M}$ over $\f{A}$.
\end{lemma}

\begin{proof}
	The statement that if there is some $g$ such that $g$ strongly embeds $\f{D}$ into $\f{M}$ over $\f{A}$, then $\dim(\f{M/A})\geq\delta(\f{D/A})$ is immediate from the definition. Thus we prove the converse. Let $\f{A}\leq\f{M}$ be finite. Let $\f{D}\in\Kfin$ be such that $\f{A\leq{D}}$.
	
	First assume that  $\delta(\f{D/A})=0$. Now as $\Sfin\models{\forall\z{x}\exists\z{y}(\Delta_{\f{A}}(\z{x})\wedge\Delta_{\f{A},\f{D}}(\z{x},\z{y}))}$. Thus there is some $\f{A\subseteq{D'}\subseteq M}$ such that $\f{D} \cong_{\f{A}} \f{D'}$. Further as $\delta(\f{D'/A})=0$, from $(2)$ of Lemma \ref{lem:CompleteTypesOverSets}, $\f{D'\leq M}$. Thus regardless of the value of $\dim(\f{M/A})$, if $\delta(\f{D/A})=0$ then there is some $g$ such that $g$ strongly embeds $\f{D}$ into $\f{M}$ over $\f{A}$. 
	
	Now assume that $m/c=\delta(\f{D/A})\leq{\dim(\f{M/A})}$ with $m \geq 1$ and further assume that $\dim(\f{M/A})\geq k/c$ with $k\geq m$. Let $\f{A}\leq\f{F}\leq{\f{M}}$ be such that $\delta(\f{F/A})=k/c$. Let $\f{G}=\f{D}\oplus\f{F}$, the free join of $\f{D,F}$ over $\f{A}$. By Lemma \ref{lem:DimWitness}, there exists $\f{H}\in\Kfin$ with $\f{G\subseteq H}$ and $\f{A,D,F\leq H}$ and $\delta(\f{H/F})=0$. Since $\f{F\leq{M}}$ and $\delta(\f{H/F})=0$ we are in the setting above. So take a strong embedding $g$ of $\f{H}$ into $\f{M}$ over $\f{F}$. Clearly $g$ fixes $\f{A}$ and $\f{D}$ has the property that $g(\f{D})\leq \f{F} \leq \f{M}$ and thus $g(\f{D})\leq{\f{M}}$. 
\end{proof}

We now obtain:

\begin{thm}\label{thm:CountableSpectrum}
	Let $\f{M,N}\models\Sfin$ be countable. Now $\f{M} \cong \f{N}$ if and only if $\dim(\f{M})=\dim(\f{N})$ and $\dim(\f{M})=\infty$ if and only if $\f{M}$ is the generic for $\Kfin$. Thus there are precisely $\aleph_0$ many non-isomorphic models of $\Sfin$ of size $\aleph_0$. Further each countable model of $\Sfin$ can be built up from a subclass of $(\Kfin,\leq)$. 
\end{thm}

\begin{proof}
	Since $\delta$ is invariant under isomorphism, it immediately follows that if $\f{M} \cong \f{N}$, then $\dim(\f{M})=\dim(\f{N})$. Now from Theorem \ref{thm:CountableModelsAsGenerics}, it follows that the number of non-isomorphic countable models is at least $\aleph_0$.\\
	
	\noindent\textit{Case 1}: $\dim(\f{M})=\dim(\f{N})=k/c$ for some $k\in\omega$. Fix enumerations for $M,N$. Let $\f{A\leq M}$ with $\dim(\f{M/A})=0$. Thus $\delta(\f{A})=\dim(\f{M})=\dim(\f{N})$. Assume that we have constructed a \textit{strong embedding} $g:\f{A\rightarrow N}$. Pick $b\in{\f{N}-g(\f{A})}$, where $b$ in the enumeration corresponds to the element of $N$ with least index not in $g(A)$. Consider $icl_{\f{N}}(\{b\}\cup{g(\f{A})}) =\f{B}\leq{\f{N}}$. Now $\f{B}$ is finite. Since $g(\f{A})\leq{\f{N}}$ and $g(\f{A})=\dim(\f{N})$, it follows that $\delta(\f{B}/g(\f{A}))=0$ and ${g(\f{A})\leq{\f{B}}}$. Now as $\f{A}\cong g(\f{A})$ by Lemma \ref{lem:DimAndStrngEm}, there exists a \textit{strong embedding} $g':\f{B\rightarrow M}$ and $g'|_{g(\f{A})}=g^{-1}$. Clearly this allows us to form a back and forth system between $\f{M,N}$. 
	
	Thus all that remains to be shown is that we can find a strong embedding of $\f{A\leq M}$ where $\delta(\f{A})=\dim(\f{M})$. To see this first note that $\emptyset\leq\f{N}$. Further $\dim(\f{N}/\emptyset)=\delta(\f{A}/\emptyset)$. Thus there exists some strong embedding of $\f{A}$ over $\emptyset$ into $\f{N}$ by an application of Lemma \ref{lem:DimAndStrngEm} as required.\\
	
	\noindent\textit{Case 2}: $\f{M}\models\Sfin$ and $\dim(\f{M})=\infty$. We claim that in this case $\f{M}$ is isomorphic to the generic. Clearly $\f{M}$ has finite closures and hence condition $(1)$ of the generic is satisfied. Note that if we show that $\dim(\f{M})=\infty$ implies that for any $\f{A\leq M}$, $\dim(\f{M/A})=\infty$, then condition $(2)$ follows immediately from Lemma \ref{lem:DimAndStrngEm}. We claim that this is indeed the case. By way of contradiction, assume that there is some $\f{A\leq M}$ such that $\dim\f{(M/A)}$ is finite. Now there is some $\f{A\leq D\leq M}$ such that $\dim(\f{M/A})=\delta(\f{D/A})$. It is immediate from the definition that $\delta(\f{M/D})=0$. As $\dim(\f{M})=\infty$, fix a $\f{B\leq M}$ with $\delta(\f{B})>\delta(\f{D})$. Consider $G$, the closure of $BD$ in $M$. Now $G$ is finite and since $B,D\leq M$, $B,D\leq G$. Further $\delta(G/D)=0$ as $\dim(M/D)=0$. So $\delta(G)=\delta(D)$. But $B\leq M$, so $\delta(G/B)\geq 0$ and hence $\delta(G)\geq \delta(B)$. Thus $\delta(B)\leq \delta(D)$, a contradiction to our choice of $B$ that establishes the claim. Hence it follows that the number of non-isomorphic countable models of $\Sfin$ is $\aleph_0$.
	
	From Theorem \ref{thm:CountableModelsAsGenerics}, it follows that we can construct a countable model of a fixed dimension (the $\dim(\f{M})=\infty$ case being the generic as seen above) as the generic of a subclass of $(\Kfin,\leq)$. But as the dimension determines the countable model up to isomorphism, we obtain the result.  
\end{proof}

We now use our notion of dimension to characterize elementary embedability.  

\begin{thm}\label{thm:CountableModelChain}
	Let $\f{M, N}$ be countable models of $\Sfin$. If $\dim(\f{M})\leq \dim{\f{N}}$, then there is some elementary embedding $f:\f{M\rightarrow N}$. Thus there is an elementary chain  $\f{M}_{0}\preccurlyeq\ldots\preccurlyeq\f{M_n}\ldots\preccurlyeq\f{M}_\omega$ of countable models of $\Sfin$ with each countable model isomorphic to some element of the chain.
\end{thm}

\begin{proof}
	Let $\f{M,N}$ be countable models of $\Sfin$ with $\dim(\f{M})\leq\dim(\f{N})$. Note that if $\dim(\f{M})=\dim(\f{N})$, then by Theorem \ref{thm:CountableSpectrum}, $\f{M}\cong\f{N}$. So assume that $\dim(\f{M})<\dim(\f{N})$ and fix an enumeration $\{m_i:i\in\omega\}$. Now we have that $\dim(\f{M})<\infty$. Let $\f{A\leq M}$ be such that $\delta(\f{A})=\dim(\f{M})$. Now by Lemma \ref{lem:DimAndStrngEm}, there exists a strong embedding $f_1$ of $\f{A}$ into $\f{N}$. Let $\f{ B\leq \f{M}}$ be such that $A \{m_{i}\}\subseteq{B}$ where $i$ is the least index such that $m_{i}\notin A$. Note that as $\delta(\f{A})=\dim(\f{M})$, $\delta(\f{B})=\delta(\f{A})$. Again using Lemma \ref{lem:DimAndStrngEm}, we can extend $f_1$ to $f_2$ so that $f_2$ is a strong embedding of $\f{B}$ into $\f{N}$ \textit{over} $\f{A}$.
	
	Proceeding iteratively we can find a $\leq$ chain $\{\f{A}_i:i\in\omega\}$ such that $\f{M}=\bigcup_{i<\omega}\f{A}_i$ and $f:\f{M\rightarrow N}$ such that $f(\f{A}_i)\leq \f{N}$ for each $i\in\omega$. It is easily seen that $f$ is an isomorphic embedding. We claim that $f$ is actually an elementary embedding of $\f{M}$ into $\f{N}$. Note that given $\f{C\leq M}$ with $C$ finite, there is some $\f{A_i}$ with $\f{C\leq A}_i\leq \f{M}$. Using the transitivity of $\leq$, it easily follows that $f(\f{C})\leq \f{N}$. In particular $f(\f{M})$ is (algebraically) closed in $\f{N}$. For notational convenience we will assume that $\f{M\subseteq N}$.    
	
	Let $\psi(\z{x},\z{y})$ be an $L$ formula. Let $\z{a}\in M^{lg(\z{x})}$. Assume that $\f{N}\models\exists\z{y}\psi({\z{a},\z{y}})$. But $\psi(\z{x},\z{y})$ is equivalent to the boolean combination of chain minimal formulas, say $\Sfin\vdash\forall(\z{x})(\exists\psi(\z{x},\z{y})\leftrightarrow\bigwedge_{i<n}\varphi_{i}(\z{x},\z{y}))$ where each $\varphi(\z{x},\z{y})$ is either a chain minimal formula or the negation of a chain minimal formula. Suppose that $\z{b}\in N^{lg(\z{y})}$ is such that $\f{N}\models\psi(\z{a},\z{b})$. If $\varphi_{i}$ is a chain minimal formula then it follows that $\z{b}\in M^{lg(\z{y})}$ as $M$ is a closed set. So assume that each $\varphi_i$ is the negation of a chain minimal formula. Note that we may split $\z{b}=\z{b}_1\z{b}_2$ where $\z{b}_1$ is formed via a minimal chain and $A\z{b}_1\leq N$. As above, it follows that $\z{b}_1\subseteq \in M^{lg(\z{y})-lg(\z{b}_1)}$. But as $\f{M}\models\Sfin$, it follows that there exists a $\z{b'}_2 \in M^{lg(\z{y})-lg(\z{b}_1)}$ that is isomorphic to $\z{b}_2$ over $A\z{b}_1$. It is now easily seen that the $\z{b}_1\z{b'}_2\in M^{lg(\z{y})}$ and $\f{N}\models\varphi_i(\z{a},\z{b}_1\z{b'}_2)$ for each $i$. Thus $\f{N}$ is an elementary extension of $\f{M}$. 
	
	Note that given an elementary chain $\f{M}_1\preccurlyeq\ldots\preccurlyeq\f{M}_n$ of models of $\Sfin$ we may construct $\f{M}_{n+1}$ such that $\f{M}_1\preccurlyeq\ldots\preccurlyeq\f{M}_n\preccurlyeq\f{M}_{n+1}$. Note that we may also insist that $\dim(\f{M}_k)=k/c$. Now given an elementary chain $\f{M}_{0}\preccurlyeq\ldots\preccurlyeq\f{M_n}\ldots\preccurlyeq$ set $\f{M}_\omega=\bigcup_{n<\omega}\f{M}_{n}$. As elementary embeddings preserve closed sets it is easily seen that $\dim({\f{M}_\omega})=\infty$. The rest of the claim now follows from Theorem \ref{thm:CountableSpectrum}.    
\end{proof}

\section{Regular Types}\label{sec:RegularTypes}

In Section \ref{sec:RegularTypes} we turn our attention towards the study of regular types. We fix a monster model $\mathbb{M}$ of $\Sfin$. Recall the notions of  $d(A)$ and $d(B/X)$ for some finite $A\subseteq \mathbb{M}$ and $X\subseteq \mathbb{M}$ from Definition \ref{defn:RelDimFunc}. We begin by extending this notion to a type as follows (see also \cite{BdShelah})

\begin{defn}
	Let $\mathbb{M}$ be a monster model of $\Sfin$ and let $X$ be a small subset of $\mathbb{M}$. Let $p\in S(X)$. We let $d(p/X)=d(\z{b}/X)$ for some (equivalently any) realization $\z{b}$ of $p$. 
\end{defn}

Now, due to $\omega$-stability and weak elimination of imaginaries (see $(3)$ and $(5)$ of Theorem \ref{thm:ForkOvClsdSets1}), it suffices to restrict our attention to non-algebraic types over finite algebraically closed sets in the home sort for the study of regular types. So fix some finite $A\leq\mathbb{M}$ (recall that algebraically closed sets are precisely the intrinsically closed ones). In what follows we freely use regular types, orthogonality, modular types etc. and facts about them. The relevant definitions and results can be found in \cite{Pil2}.

\begin{remark}\label{rmk:SettingNFOmgSt}
	Let ${A\leq \mathbb{M}}$ be finite and $\z{b}$ be finite such that $\z{b}\cap{A}=\emptyset$. Now let ${A\subseteq{C}}$ also be finite. Note that $\z{b}\dnf_{{A}}{C} $ if and only if $\acl(\z{b}{A})\dnf_{\acl({A})}\acl({C})$. Since $\Sfin$ has finite closures it follows that $\acl({b}{A}), \acl(C)$ are both finite. Thus in order to understand non-forking, it suffices to look at types $p\in S({A})$ such that $x\neq a \in p$ for all $a\in{{A}}$ such that for any $\z{b}\models p$, $\z{b}{A}\leq \mathbb{M}$. Note that this information, along with the atomic diagram of some (of any) realization of $p$ is sufficient to determine $p$ uniquely as noted in $(1)$ of Lemma \ref{lem:CompleteTypesOverSets}.  Also such a type $p$ is non-algebraic and stationary as ${A}$ is algebraically closed.
\end{remark}

In light of our comments at the beginning of Section \ref{sec:RegularTypes} and Remark \ref{rmk:SettingNFOmgSt} it suffices to study \textit{basic types over finite sets} in order to understand regular types (i.e. we can choose a basic type to represent the required parallelism class).  

\begin{defn}\label{defn:basicType}
	Let $A\leq \mathbb{M}$ be finite and $p\in S(A)$, we say that $p$ is a \textit{basic type} if $x\neq a \in p$ for all $a\in{A}$ and for some (equivalently any) $\z{b}\models p$, $\z{b}{A}\leq {\mathbb{M}}$. 
\end{defn}

\begin{remark}
	Note that if $\f{A,B}\in K_L$ with $\f{A}\in\Kfin$ and $\f{A\leq B}$, then $\f{B}\in\Kfin$.
\end{remark}

\begin{lemma}\label{lem:ExistenceStngExt}
	Let $\f{A}\in \Kfin$. Then there exists $\f{B}\in\Kfin$ such that $\f{A\leq{B}}$ and $\delta(\f{B/A})=1/c$. 
\end{lemma}

\begin{proof}
	Consider the structure given by $\f{A}^*=\f{A}\oplus{\f{A}_0}$ where $\f{A}_0\in\Kfin$ consists of a single point. Now an application of Lemma \ref{lem:FinStrsInRatCase} to $\f{A^*}$ yields the required result.
\end{proof}

We begin by studying basic types such that $d(p/{A})= 0, 1/c$ where ${A\leq \mathbb{M}}$ is finite. The choice to restrict our attention to such types will be justified by Theorem \ref{lem:2/cHasPrWrGEQ2}, where we show any type $p$ with $d(p/A)\geq 2/c$ \textit{cannot} be regular. We begin our analysis of types that can be regular types by defining nuggets and nugget-like types.  

\begin{defn}\label{defn:nugget}
	Let $\f{A,D}\in\Kfin$ with $\f{A\subsetneq{D}}$ with $D=AB$. Let $k\in \omega$. We say that ${B}$ is a $k/c$-\textit{nugget over} $\f{A}$ if $A\cap B=\emptyset$, $\delta({B/A})=k/c$ and $\delta({B'/A})>k/c$ for all ${A\subsetneq{AB'}\subsetneq AB}$.
\end{defn}


\begin{defn}\label{defn:nuggetlike-type}
	Let $A\leq\mathbb{M}$ be finite. We say that a basic type $p\in{S({A})}$ is \textit{nugget-like} over ${A}$, if given ${B}$ where ${B}$ realizes the quantifier free type of $p$ over $\f{A}$, then ${B}$ is a $k/c$-nugget over ${A}$ for some $k\in\omega$. 
\end{defn}

\begin{lemma}\label{lem:IntersecOfNuggets}
	Let $A\leq \mathbb{M}$ be finite and let $p\in S(A)$ be nugget-like. Let $A\subseteq X$ with $X$ closed. For any $\z{b}\models p$, either $\z{b}\cap{X}=\emptyset$ or $\z{b}\subseteq {X}$ . 
\end{lemma}

\begin{proof}
	Assume that $\z{b}\cap{X}\neq \emptyset$. Let $\z{b}'=\z{b}\cap{X} $ assume that $\z{b}'\neq \z{b}$. Then as $\delta(\z{b}'/A)>\delta(\z{b}/A)$, it follows that there is some minimal pair $(A\z{b}',D)$ with $D\subseteq A\z{b}$ but $D\nsubseteq X$. But this contradicts that $X$ is closed. Hence $\z{b}\subseteq X$.    
\end{proof}

We now explore how the behavior of the $d$ function interacts with nugget-like types. The following results are well known (see e.g. Theorem 3.28 of \cite{BdSh} or Lemma 3.13 of \cite{VY} and Lemma 2.6 of \cite{BdShelah}). 

\begin{lemma}\label{prop:dMonoBase}\label{lem:forkingDropsDimensions}\begin{enumerate}
		\item Suppose ${B}$ is finite and ${X\subseteq Y}$. Then $d({B/X})\geq{d({B/Y})}$.
		\item Let $A\leq\mathbb{M}$ be finite and let $p\in{S({A})}$. Suppose that for some $k\in\omega$, $d(p/{A})=k/c$. Let ${A\subseteq X\leq \mathbb{M}}$. Suppose that $q\in S({X})$ extends $p$. If $d(q/{X})<d(p/{A})$, then $q$ is a forking extension of $p$.
	\end{enumerate}
	
\end{lemma}

We now obtain the following fact about the forking of nugget-like types:

\begin{lemma}\label{lem:CharOfForking1}
	Let $A\leq\mathbb{M}$  be finite and let $p\in S({A})$ is nugget-like. Let ${A\subseteq Y\subseteq \mathbb{M}}$ with $Y$ closed. Let $q$ be an extension of $p$ to $Y$. Now $q$ is a forking extension of $p$ if and only if $d(q/{Y})<d(p/{A})$ or given $\z{b}\models q$, $\z{b}\subseteq Y$. 
\end{lemma}

\begin{proof}
	If $d(q/{Y})<d(p/{A})$, then Lemma \ref{lem:forkingDropsDimensions} tells us that $q$ is a forking extension of $p$. Further $Y$ is algebraically closed. So if for any $\z{b}\models q$, $\z{b}\subseteq Y$, it follows that $b$ is an algebraic type over $Y$. Since $p$ is not an algebraic type over $A$, it follows that $q$ is a forking extension of $p$. 
	
	For the converse assume that $q$ is a forking extension of $p$ and that $d(q/{Y})=d(p/{A})$. As $q$ is a forking extension of $p$, it follows from $(8)$ of Theorem \ref{thm:ForkOvClsdSets} that $\icl(\z{b}A)\cap\icl({Y})\supsetneq \icl(A)$. But  $\icl(A)=A$, $\icl(Y)=Y$ and as $\z{b}$ realizes $p$ over $A$, $\icl(\z{b}A)=\z{b}A$. Thus $\z{b}\cap{Y}\neq\emptyset$.  Now by Lemma \ref{lem:IntersecOfNuggets}, $\z{b}\subseteq Y$. 
\end{proof}

The following theorem allows us to identify certain regular types. Further it establishes that $0$-nuggets are, in some sense, orthogonal to almost all other types.

\begin{thm}\label{thm:SomeNuggetsAreRegular}
	Let $A\leq\mathbb{M}$ be finite and let $p\in S({A})$ be nugget-like. Now if $d(p/{A})=0$ or $d(p/{A})=1/c$, then $p$ is regular. Further if $d(p/A)=0$, then $p$ is orthogonal to any other nugget-like type over $A$.  
\end{thm}

\begin{proof}
	Under the given conditions $p$ is clearly non-algebraic and stationary. We directly establish that it will be orthogonal to any forking extension of itself. Let ${A\subseteq{X}\subseteq \mathbb{M}}$ with $X$ closed. Since $\Sfin$ is $\omega$-stable and has finite closures we may as well assume that $X$ is finite, i.e. if $q\in S(X)$ with $q\supseteq p$ a forking extension, there is some finite closed $X_0\subseteq X$ such that $q\upharpoonright_{X_0}$ is a forking extension. Let $\z{b}\models p$. We have that $\z{b}\dnf_{A}X$. As $A\z{b}$, $X$ are closed and $A\z{b}\cap{X}=A$, from an application of $(8)$ of Theorem \ref{thm:ForkOvClsdSets}, we obtain that $X\z{b}$ is closed. 
	
	First assume that $d(p/{A})=0$. Let $p'$ be a forking extension of $p$ to ${X}$ and let$\z{f}\models p'$. It follows easily from Lemma \ref{prop:dMonoBase}, that  $d(\z{f}/A)\geq d(\z{f}/X)$. As $d(\z{f}/A)=0$ and $d(\z{f}/X)\geq 0$, it now follows that $d(\z{f}/X)=0$. Thus by Lemma \ref{lem:CharOfForking1}, we have that $\z{f}\subseteq X$ and hence $\z{b}\dnf_{X}\z{f}$ as $\z{b} \dnf_{A} X$.
	
	So assume that $d(p/{A})=1/c$. Let $p'$, $\z{f}$ be as above. By Lemma \ref{lem:CharOfForking1}, $d(p'/{X})=0$ or $\z{f}\subseteq{X}$. As above $\z{f}\subseteq X$ yields that $\z{b}\dnf_{X}\z{f}$. So assume that $\z{f}\nsubseteq X$ and note that by Lemma \ref{lem:IntersecOfNuggets} we have that $\z{f}\cap X=\emptyset$. Now by $(8)$ of Theorem \ref{thm:ForkOvClsdSets} it suffice to show that $X\z{b}\cap \acl(X\z{f})=X$ to establish that $\z{b}\dnf_{X}\z{f}$. Consider $d(\acl(X\z{f})\,\z{b}/X)$. On the one hand we have that $d(\acl(X\z{f})\,\z{b}/X)\geq d(\z{b}/X)=1/c$ (see Remark \ref{rmk:propOfd}). On the other hand $d(\acl(X\z{f})\,\z{b}/X)=d(\z{b}/acl(X\z{f}))+d(\acl(X\z{f})/X)$. As $d(\acl(X\z{f})/X)=d(\z{f}/X)=0$, we obtain that $d(\z{b}/\acl(X\z{f}))\geq 1/c$. In particular $\z{b}\nsubseteq\acl(X\z{f})$. But then by Lemma \ref{lem:IntersecOfNuggets}, $\z{b}\cap\acl(X\z{f})=\emptyset$ and thus $X\z{b}\cap\acl(X\z{f})=\emptyset$ as required. 
	
	For the second half of the claim, assume that $d(p/A)=0$. Let $q\in S(A)$ be nugget-like and distinct from $p$. Now $d(p/A)=d(p|_{X}/X)$ and $d(q/A)=d(q|_{X}/X)$. Let $\z{f}\models q|_{X}$. Note that $\z{f}\dnf_{A}X$ implies that $X\z{f}$ is closed. Now using Lemma \ref{lem:IntersecOfNuggets}, we can easily show that $\z{b}X\cap{\z{f}}X\neq X$, then $\z{b}=\z{f}$. But this contradicts $p\neq q$. Thus it follows that $\z{b}X\cap\z{f}X=X$. Further $0=d(\z{b}/X)\geq d(\z{b}/X\z{f})\geq 0$. Again by $(8)$ of Theorem \ref{thm:ForkOvClsdSets}, we obtain that $\z{b}\dnf_{X}\z{f}$  and thus $p,q$ are orthogonal.   
\end{proof}

The following theorem shows that while there are many regular types with $d(p/A)=1/c$, all such types are non-orthogonal. Thus up to non-orthogonality, there is only one regular type with $d(p/A)=1/c$. This is in contrast to distinct $0$-nuggets, any two of which are orthogonal to each other. We also show that the number of independent realizations of a $1/c$ nugget determines the dimension of a model.

\begin{thm}\label{thm:DomEquivalence}
	Let $A$ be closed and finite and let $p,q\in S(A)$ be distinct basic types and satisfy $d(p/A)=d(q/A)=1/c$. Then $p,q$ are non-orthogonal. Hence any two regular types over $p',q'\in S(X)$ where $X$ is closed and $d(p'/X)=d(q'/X)=1/c$ are non-orthogonal. Further if we take $A=\emptyset$ and let $\f{M\preccurlyeq\mathbb{M}}$. The dimension of $\f{M}$ is determined by the number of independent realizations of $p$ in $\f{M}$. Thus a single regular type determines the dimension of $\f{M}$.
\end{thm} 

\begin{proof}
	Let $A$ be as given. Consider $A$ as a finite structure that lives in $\Kfin$. Now consider the finite structures $A{B}, A{C}$ where $B, C$ realize the quantifier free types of  $p,q$ respectively. Consider $D$, the free join of $AB,AC$ over $A$. Apply Lemma \ref{lem:FinStrsInRatCase} to obtain a finite  $G$ with $\delta(G/D)=-1/c$ and $A,AB,AC\leq G$. Let $f$ be a strong embedding of ${G}$ into $\mathbb{M}$ where $f$ is the identity on $A$. From $(6)$ of Theorem \ref{lem:CompleteTypesOverSets} and the transitivity of $\leq$ it follows that $f(B)\models p$ and $f(C)\models q$. Now from $(8)$ of Theorem \ref{thm:ForkOvClsdSets}, it follows that $f(B)\forks_{A}f(C)$ and thus $p\north q$. Now given $p',q'\in S(X)$, there exists a finite closed set, which by an abuse of notation we call $A$, such that $p',q'$ are based and stationary over $A$. Since regularity is parallelism invariant both $p|_{A}$ and $q|_{A}$ are regular. Arguing as above we see that $p'|_{A}\north q'|_{A}$ and thus they are non-orthogonal. 
	
	Let $\f{M}\preccurlyeq\mathbb{M}$ and assume that $A=\emptyset$. Given ${n\in\omega}$, consider the finite structure ${C}_n$ that is the free join of $n$-copies of the quantifier free type of $p$ over $\emptyset$. If $\dim(\f{M})\geq n/c$, by Lemma \ref{lem:DimAndStrngEm}, there is a strong embedding of $C_n$ into $\f{M}$. It is easily checked that the strong embedding witnesses $n$-independent realizations of $p$. The rest follows easily. 
\end{proof}

The following shows that $1/c$ nuggets are not locally modular.

\begin{thm}\label{thm:NuggetsAreNonLocMod}
	Let $A\leq \mathbb{M}$ be finite and let $p\in S(A)$ be a nugget-like with $d(p/A)=1/c$. Then $p$ is not locally modular, in particular it is non-trivial. 
\end{thm}

\begin{proof}
	Recall that given a regular type $p$, the realizations of $p$ form a pregeometry with respect to forking closure. In order to simplify the presentation, we will let $A=\emptyset$. We let $p^\mathbb{M}$ denote the realizations of $p$ in $\mathbb{M}$, $\pcl$ denote the forking closure (or $p$-closure) of $p^\mathbb{M}$ and  $\pdim$ ($p$-dimension) denote the associated dimension.
	
	We begin with a proof that $p$ is non-trivial. Let $B_1,B_2,B_3$ be three finite structures that has the same quantifier free type as $p$ and are disjoint over $\emptyset$. Consider  ${C}= \oplus {B}_i$, the free join of the $B_i$ over $\emptyset$. Using Lemma \ref{lem:FinStrsInRatCase} we obtain a finite structure $D\in \Kfin$ with $\delta(D)=2/c$, $B_i\leq C$ and $B_i\oplus B_j\leq C$ for any $i\neq j$. Note that $C\nleq D$ as $\delta(C)>\delta(D)$. Let $g$ be a strong embedding of $C$ into $\mathbb{M}$. An argument similar to that found in Theorem \ref{thm:DomEquivalence} shows that $g(B_1), g(B_2), g(B_3)$ are pairwise independent but dependent realizations of $p$ and thus $p$ is non-trivial.

	We will now establish that $p$ is not modular. To show that $p$ is not locally modular, we can simply choose a realization $\z{h}$ of $p$ independent from the configuration used in the following argument  and relativize the argument over $\z{h}$. Fix realizations $\z{a},\z{b},\z{c}\models p$ such that they are pairwise independent but are dependent. As $\Sfin$ is stable we can find $\z{b'}\, ,\z{c'} \models p$ such that $\z{b'}\,\z{c'}\equiv_{\z{a}}\z{b}\,\z{c}$ and $\z{b'}\,\z{c'}\dnf_{\z{a}}\z{b}\,\z{c}$. Let $X=\pcl(\z{c'}\,\z{b})$, $Y=\pcl(\z{c}\,\z{b'})$. Let $Z=\pcl(X\cup Y)$. We will show that $\pdim(\pcl(Z))+\pdim(X\cap{Y})<\pdim(X)+\pdim(Y)$. As $\pdim(X)=2=\pdim(Y)$ and $\pdim(Z)=3$, it suffices to show that $X\cap Y \cap{p^\mathbb{M}}$ is empty.

	Suppose to the contrary that there is some $\z{e}\in X\cap{Y}\cap{p^\mathbb{M}}$. We obtain the following configuration: 
	
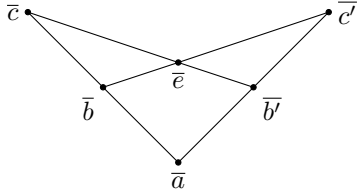
\begin{center}\begin{tikzpicture}
		\draw (0,0) --(1,1);
		\draw (0,0) --(-1,1);
		\draw (1,1) --(2,2);
		\draw (-1,1) --(-2,2);
		\draw (-1,1) --(0,1.33);
		\draw (1,1) --(0,1.33);
		\draw (2,2) --(0,1.33);
		\draw (-2,2) --(0,1.33);
		\node[below] at (0,0) {$\z{a}$};
		\filldraw[black] (0,0) circle (1pt);
		\node[below right] at (1,1) {$\z{b'}$};
		\filldraw[black] (1,1) circle (1pt);
		\node[below left] at (-1,1) {$\z{b}$};
		\filldraw[black] (-1,1) circle (1pt);
		\node[left] at (-2,2) {$\z{c}$};
		\filldraw[black] (-2,2) circle (1pt);
		\node[right] at (2,2) {$\z{c'}$};
		\filldraw[black] (2,2) circle (1pt);
		\node[below] at (0,1.33) {$\z{e}$};
		\filldraw[black] (0,1.33) circle (1pt);
		\end{tikzpicture}
	\end{center}
	
	\noindent with every single (labeled) point having $p$-dimension $1$, any three (labeled) colinear points (as found in the configuration) having $p$-dimension $2$ and any three (labeled) non-colinear points having $p$-dimension $3$.   
	
	Let $C^*=\acl(\z{a}\,\z{b}\,\z{c}\z{b'}\,\z{c'}\,\z{e})$. We will obtain a contradiction by estimating $d(C^*)$ in two different ways. On the one hand, as $\{\z{a} ,\z{b} ,\z{b'}\}$ is independent and $\acl(\z{a}\,\z{b}\,\z{b'})\subseteq C^*$, it follows that $d(C^*)\geq 3/c$ (recall that $d(C^*/\acl(\z{a}\,\z{b}\,\z{b'}))\geq 0$). On the other hand $\acl(C_1C_2C_3C_4)=C^*$ where $C_1=\acl(\z{c}\,\z{e}\,\z{b'})$, $C_2=\acl(\z{b}\,\z{e}\,\z{c'})$, $C_3=\acl(\z{a}\,\z{b}\,\z{c})$ and $C_4=\acl(\z{a}\,\z{b'}\,\z{c'})$. We estimate $d(C_1C_2C_3C_4)$ using \textit{flatness} (see $(9)$ of Theorem \ref{thm:ForkOvClsdSets}). 
	
	We begin by showing $d(C_1)=2/c$.  Note that as $\z{b}\dnf \z{c'}$, by $(8)$ of Theorem \ref{thm:ForkOvClsdSets} we obtain that $\z{b}\,\z{c'}$ is closed. As $\z{e}\forks\z{c'}\,\z{b}$ another application of $(8)$ of Theorem \ref{thm:ForkOvClsdSets} tells us that $\z{e}\cap{\z{b}\,\z{c'}}$ is non-empty or that $d(\z{e}/\emptyset)>d(\z{e}/\z{b}\,\z{c'})$. An application of Lemma \ref{lem:IntersecOfNuggets} easily shows that $\z{e}\cap{\z{b}\,\z{c'}}$ is empty and hence $d(\z{e}/\emptyset)>d(\z{e}/\z{b}\,\z{c'})$. But as $d(\z{e}/\emptyset)=d(\z{e})=1/c$ we obtain that $d(\z{e}/\z{b}\,\z{c'})=0$. But then $d(C_1)=d(\z{e}\,\z{b}\,\z{c'})=d(\z{b}\,\z{c'})=2/c$ as $\z{b}\,\z{c'}$ is closed.  Similar arguments show that $d(C_i)=2/c$ for $i=2,3,4$.  
	
	We now claim that $d(C_i\cap C_j)=1/c$ for each $1\leq i < j \leq 4$. Fix $1\leq i < j\leq 4$. Note that as $C_i\cap C_j$ contains a realization of $p$, and hence $d(C_i\cap{C_j})\geq 1/c$. Further $C_i\cup C_j$ contains three independent realizations of $p$ and hence $d(C_i C_j)\geq 3/c$. Using flatness on $I'=\{i,j\}$, we obtain that $d(C_iC_j)\leq d(C_i)+d(C_j)-d(C_i\cap C_j)$ and the claim follows. A similar argument shows that $d(C_i\cap C_j\cap C_k)=0$ for $1\leq i <j <k \leq 4$. Now as $C_1\cap C_2\cap C_3\cap C_4 \subseteq C_1\cap C_2\cap C_3$ and $d(C_1\cap C_2\cap C_3)=0$, it follows that $d(C_1\cap C_2\cap C_3\cap C_4)=0$. 
	
	Hence we obtain that $3/c\leq d(C_1 C_2 C_3 C_4)\leq 4(2/c)-6 (1/c)-4(0)+0 = 2/c$, a contradiction which shows that $X\cap Y \cap p^{\mathbb{M}}$ is empty. Thus $p$ is not modular.     
\end{proof}

\begin{remark}\label{rmk:non-localmodularity}
	We sketch an alternate proof of Theorem \ref{thm:NuggetsAreNonLocMod} that uses only the non-triviality of $p$: Well known results of Hrushovski in \cite{Hru2} state that any stable theory with a non-trivial locally modular regular type interprets a group. As these structures do not interpret groups (see \cite{FW} by Wagner) the result now follows. 
\end{remark}

The following result shows that a broad class of types cannot be regular types and justifies the choice to study types $p\in S(A)$ with $d(p/A)=0,1/c$ in our study of regular types.  

\begin{thm}\label{lem:2/cHasPrWrGEQ2}
	Let ${A}$ be finite and closed in $\mathbb{M}$. Let $p\in S({A})$ be a basic type such that $d(p/{A})\geq 2/c$. Then $p$ is not regular. 
\end{thm}

\begin{proof}
	Recall that a regular type has weight 1. We establish the above result by showing that $p$ has pre-weight at least 2 and hence weight at least 2. Our strategy is similar to the one used in Theorem \ref{thm:DomEquivalence}: we consider $A$ as living inside of $\Kfin$. We then construct a finite structure ${G}$ over the finite structure ${A}$ that we then embed strongly into $\mathbb{M}$ over ${A}$ using saturation. Finally we argue that the strong embedding witnesses the fact that the pre-weight of $p$ is at least 2. 
	
	Consider $A$ as a finite structure that lives in $\Kfin$. By Lemma \ref{lem:ExistenceStngExt} we may construct ${D}\in\Kfin$ such that the $D=AC$, $A\cap{C}=\emptyset$ (as sets) and ${A\leq D}$ with $\delta(D/A)=\delta({C/A})=1/c$. Let $AB$ be such that ${B}$ realizes the quantifier free type of $p$ over $A$. Consider the finite structures $F_i$, $i=1,2$ where each $F_i$ is the free join of $AB$ and an isomorphic copy of $D$ over $A$ and $F_1\cap F_2=AB$. We label the isomorphic copies of $D$ as ${AC}_1,{AC}_2$ and thus $F_i=ABC_i$, the free join of $AB,AC_i$ over $A$. Apply Theorem \ref{lem:InfMinPa} to obtain $G_i$ for $i=1,2$ such that $(F_i,G_i)$ is an essential minimal pair and $\delta(G_i/F_i)=-1/c$. It is easily verified that $A,AB,AC_i\leq G_i$. Let ${G}$ be the free join of ${G}_1,{G}_2$ over ${AB}$. Note that ${G}\in K_L$ and that we may now regard the finite structures $A,AB,AC_1$ etc. as substructures of $G$.  
	
	We claim that ${G}\in\Kfin$, ${A},{AB},{AC}_1,{AC}_2, {AC}_1{C}_2\leq {G}$ but  ${F}_1{F}_2$, the free join of $ABC_1,ABC_2$ over $AB$ is \textit{not} strong in $G$. Using Remark \ref{rmk:ConstrVerifSimpl} and  the transitivity of $\leq$, we obtain that it suffices to show that $AB, AC_1C_2\leq G$ along with $F_1F_2\nleq G$ to obtain the claim. 
	
	First, as $AB\leq G_i$ and $G$ is the free join of $G_1, G_2$ over $AB$, we obtain $AB\leq G$ by an application of $(2)$ of Fact \ref{fact:rankAddOverBases}. We now show that ${AC}_1{C}_2\leq {G}$. Let ${AC}_1{C}_2\subseteq{G}'\subseteq{G}$ and let ${B'}={B\cap{G'}}$, ${G}_i'={G}_i-{AC}_i$. Now $\delta({G'}/{AC}_1{C}_2)=\delta(({G}_1'-{B'})({G}_2'-{B'})/{AC}_1{C}_2{B}')+\delta({B'}/{AC}_1{C}_2)$ using $(3)$ of Fact \ref{lem:BasicDel3}. Further, since $AB,AC_1C_2$ is freely joined over $A$ $\delta({B'}/{AC}_1{C}_2)=\delta({B'}/{A})$ follows from $(1)$ of Fact \ref{lem:MonoRelRankOvBase}. Arguing similarly we obtain that $\delta({G}_i'-{B'}/{AC}_1{C}_2{B}')=\delta({G}_i'/{AB'C}_i)$. Thus it follows that
	$\delta({G'}/{AC}_1{C}_2)=\delta({G}_1'/{AC}_1{B}') + \delta({G}_2'/{AC}_2{B}') +\delta({B'}/{A})$. Now as $A\leq AB$, it follows that $\delta(B'/A)\geq 0$. The claim now follows by considering the cases $B'\neq B$ and $B'=B$ using that fact that $(ABC_i,G_i)$ forms an essential minimal pair. Finally, and easy calculation  shows that $\delta({G}/{F}_1{F}_2)=-2/c$. 
	
	Arguing as we did in Theorem \ref{thm:DomEquivalence}, we easily obtain that a strong embedding of $G$ into $\mathbb{M}$ over $A$ witnesses that the pre-weight of $p$ is at least 2. We leave the details to the reader.
\end{proof}

\section{A pseudofinite $\omega$-stable theory with a non-locally modular regular type}\label{sec:OtherProp}

In this section we draw on some known results to prove that there are pseudofinite $\omega$-stable theories with non-locally modular regular types. This answers a question of Pillay's in \cite{Pil1} regarding whether pseudofinite stable theories always have locally modular regular types. We assume that the reader is familiar with basic facts about pseudofinite theories. 

\begin{thm}\label{thm:Example}
	There is a pseudofinite $\omega$-stable theory with a non-locally modular regular type.   
\end{thm} 

\begin{proof}
	Consider the case where $L=\{E\}$ contains only one relation symbol (of arity at least $2$) and let $\z{\alpha}\in (0,1)$ be rational. We claim that $\Sfin$ has the required properties.
	
	Let $\{\alpha_n\}$ be an increasing sequence of irrationals in $(0,1)$ that converge to $\z\alpha$. By the results of \cite{BdShelah1}, it follows that $Th(\f{M}_{\alpha_n})$ can be obtained as a almost sure theory with respect to a certain probability measure. Thus, in particular, each theory $Th(\f{M}_{\alpha_n})$ is pseudofinite. Now by Theorem 4.2 of \cite{BdLas}, it follows that $\Sfin = Th(\Pi_{\mathcal{U}} \f{M}_{\alpha_n})$ where $\mathcal{U}$ is a non-principal ultrafilter on $\omega$. Since taking the ultraproduct of structures with pseudofinite theories results in a structure with a pseudofinite theory, it follows that $\Sfin$ is pseudofinite. Further as we have shown in Theorem \ref{thm:NuggetsAreNonLocMod} that $1/c$-nuggets are non-locally modular and the result follows.
\end{proof}

\bibliographystyle{plain}
\bibliography{CountMod}
\Addresses
\end{document}